\documentclass[english, a4paper, 12pt]{article}
\usepackage{multicol,amsmath,t1enc,a4,babel}
\usepackage{amsthm}
\usepackage{amssymb}
\usepackage[dvips]{graphicx}
\usepackage[dvips]{epsfig}
\usepackage{times}
\usepackage{amsfonts}
\usepackage{wasysym}
\usepackage{mathtools}
\usepackage[colorinlistoftodos]{todonotes}
\usepackage{mathrsfs}
\usepackage{relsize}
\usepackage{bbm}

\newcommand{\e}{\mathbb{E}}

\newcommand{\p}{\mathbb{P}}

\newcommand{\boh}{{\hat{\gamma}_T}}

\DeclareMathOperator{\vect}{vec}

\newcommand{\bh}{{H}}
\newcommand{\Bh}{{B}}
\newcommand{\Ch}{{C}}
\newcommand{\Dh}{{D}}

\newtheorem{defi}{Definition}[section]
\newtheorem{lemma}[defi]{Lemma}

\newtheorem{theorem}[defi]{Theorem}
\newtheorem{kor}[defi]{Corollary}

\newtheorem{rem}[defi]{Remark}
\newtheorem{ex}[defi]{Example}

\begin{document}

\title{Modeling and estimation of multivariate discrete and continuous time stationary processes}

\renewcommand{\thefootnote}{\fnsymbol{footnote}}

\author{Marko Voutilainen}

\maketitle

\begin{abstract}
In this paper, we give a AR$(1)$ type of characterization covering all multivariate strictly stationary processes indexed by the set of integers. Consequently, we derive continuous time algebraic Riccati equations for the parameter matrix of the characterization providing us with a natural way to define the corresponding estimator under the assumption of square integrability. In addition, we show that the estimator inherits consistency from autocovariances of the stationary process and furthermore, the limiting distribution is given by a linear function of the limiting distribution of the autocovariances. We also present the corresponding existing results of the continuous time setting paralleling them to the discrete case.
\end{abstract}

{\small
\medskip


\medskip

}



\section{Introduction}

Stationary stochastic processes provide a significant instrument for modeling numerous temporal phenomena related to different fields of science. In particular, due to the evidence of long dependence structures in the real financial data, stationary processes possessing long-memory have been widely applied in mathematical finance. 

When discrete time is considered, stationary data is typically modeled by applying ARMA processes or their extensions. One focal reason for popularity of ARMA processes is that for every stationary process with a vanishing autocovariance $\gamma(\cdot)$ and for every $n\in\mathbb{N}$ there exists an ARMA process $X$ such that $\gamma_X(k) = \gamma(k)$ for $|k| \leq n$. For a comprehensive overview of ARMA processes we mention \cite{brockwell1991time}, \cite{hamilton1994time} and \cite{neusser2016time}. The immense ARMA family include for example SARIMA models, where a seasonal ARMA process is obtained by differencing the original data, and different GARCH models originating from \cite{engle1982autoregressive} and \cite{bollerslev1986generalized} that are commonly used in financial modeling taking account of the time-dependent volatility. ARMA processes, their extensions and estimation have been concerned e.g. in \cite{hannan1973asymptotic}, \cite{tiao1983consistency}, \cite{mikosch1995parameter}, \cite{mauricio1995exact}, \cite{ling2003asymptotic}, \cite{francq2004maximum}, \cite{bollerslev2008glossary}, \cite{han2014asymptotic}, \cite{baillie1996fractionally}, \cite{ling1997fractionally} and \cite{zhang2001nonlinear}, to mention but a few. Moreover, in \cite{voutilainen2017model} we showed that all univariate strictly stationary processes indexed by the integers are characterized by the AR$(1)$ equation
\begin{equation*}
X_t - \phi X_{t-1} = Z_t,\quad t\in\mathbb{Z},
\end{equation*}
where the noise $Z$ belonging to a certain class of stationary process is not necessarily white. Established on the characterization, we proposed an estimation method for $\phi$ in the case of a square integrable stationary process that has several advantages over the conventional methods such as maximum likelihood and least squares fitting of ARMA models. Furthermore, in \cite{arch} we applied our method in estimation of a generalized ARCH model involving a covariate process that can be interpreted as the liquidity of an asset.

In the case of continuous time, the Ornstein-Uhlenbeck process $X$ given by the Langevin equation
\begin{equation}
\label{introlangevin}
dX_t = -\theta X_t dt + dB_t, \quad t\in\mathbb{Z},
\end{equation}
where $\theta >0$ and $B$ is a two-sided Brownian motion, can be seen as the analogue of the discrete time AR$(1)$ process. By posing a suitable initial condition, \eqref{introlangevin} yields a stationary solution. The foregoing can be generalized, for example, by replacing Brownian motion with other stationary increment processes satisfying certain integrability conditions, such as fractional Brownian motion recovering the fractional Ornstein-Uhlenbeck process introduced in \cite{Cheridito-Kawaguchi-Maejima-2003}. This kind of generalized Ornstein-Uhlenbeck processes are applied e.g. in mathematical finance to describe mean-reverting systems under the influence of shocks, and they are a highly active topic of research. Equations of type \eqref{introlangevin} with varying driving forces, and estimation in such models have been concerned e.g. in \cite{Hu-Nualart-2010}, \cite{Kleptsyna-LeBreton-2002}, \cite{Azmoodeh-Viitasaari-2015a}, \cite{bajja2017least}, \cite{balde2018ergodicity}, \cite{Brouste-Iacus-2013}, \cite{Ciprian-Hermite}, \cite{hu2019parameter}, \cite{nourdin2019statistical}, \cite{Sottinen-Viitasaari-2017a}, \cite{tanaka2015maximum}, \cite{applebaum2015infinite} and \cite{magdziarz2008fractional}, to mention but a few.
Furthermore, in \cite{voutilainen2019vector} we showed that a generalized multidimensional version of \eqref{introlangevin} characterizes all multivariate strictly stationary processes with continuous paths. Consequently, we proposed an estimation method for the parameter matrix of \eqref{introlangevin} under the assumption of square integrability. The method is based on continuous time algebraic Riccati equations (CAREs) written in terms of the autocovariance function of the stationary solution. Algebraic Riccati equations have been studied intensively in the literature and they occur naturally e.g. in optimal control and filtering theory. Real-valued CAREs often take the symmetric form
\begin{equation}
\label{introriccati}
B^\top A + A B - A C A + D = 0,
\end{equation}
where $C$ and $D$ are symmetric, and symmetric solutions $A$ are to be found. For a general approach to algebraic Riccati equations the reader may consult for example \cite{lancaster1995algebraic}. The existence and uniqueness of a solution to \eqref{introriccati} is a well-studied topic, especially when $C$ and $D$ are positive semidefinite (see e.g. \cite{kucera1972contribution}, \cite{wonham1968matrix} or \cite{sun1998perturbation}).

The rest of the paper is organized as follows. In Subsection \ref{sec:maindisc} we complete our previous investigations of stationary processes by treating the multivariate discrete time case. First, we show that the characterization is now given by a multidimensional AR$(1)$ type of equation. Then, by taking a similar approach as in \cite{voutilainen2019vector} we obtain a set of symmetric CAREs that serve as a basis for estimation of the model parameter matrix. Finally, we state theorems for consistency and asymptotic distribution of the estimator. In Subsection \ref{sec:maincont} we present the main results of \cite{voutilainen2019vector} while at the same time comparing them to the results obtained in discrete time. The proofs are postponed to Section \ref{sec:proofs}.

\section{Main results}
\label{sec:main}
The considered processes are $n$-dimensional, real-valued and indexed by $I \in \{\mathbb{Z}, \mathbb{R}\}$. We use the notation $Y = (Y_t)_{t\in I}$, where the $i$th component of the random vector $Y_t$ is denoted by $Y_t^{(i)}$. Equality of the distributions of two random vectors $Y_t$ and $Z_t$ is denoted by $Y_t \overset{\text{law}}{=} Z_t$. Similarly, equality of two processes $Y$ and $Z$ in the sense of finite dimensional distributions is denoted by $Y = (Y_t)_{t\in I} \overset{\text{law}}{=} (Z_t)_{t\in I} =  Z$. Throughout the paper, we investigate strictly stationary processes meaning that $(X_{t+s})_{t\in I}  \overset{\text{law}}{=} (X_t)_{t\in I}$ for every $s \in I$. Consequently, we omit the word 'strictly' and simply say that $X$ is stationary. By writing $A \geq 0$ or $A >0$ we mean that the matrix $A$ is positive semidefinite or positive definite, respectively. We denote an eigendecomposition of a symmetric matrix by $A = Q\Lambda Q^\top$, where $\Lambda= \mathrm{diag}(\lambda_i)$. Furthermore, the $L^2$ vector norm and the corresponding induced matrix norm is denoted by $\Vert \cdot \Vert$.

By applying the models of stationary processes, which we introduce in this paper, we consequently obtain symmetric CAREs of the form
\begin{equation}
\label{riccati2}
B^\top A + A B - A C A + D = 0,
\end{equation}
where $C, D \geq 0$, and we are solving the equation for a positive definite $A$. There exists a vast amount of literature on existence and uniqueness of a solution (see e.g. \cite{kucera1972contribution} or \cite{wonham1968matrix}) in the described setting. In particular, if $C, D >0$, then there exists a unique positive semidefinite solution to \eqref{riccati2}. Furthermore, there exists several numerical methods for finding the positive semidefinite solution of \eqref{riccati2} (see e.g. \cite{byers1987solving}, \cite{laub1979schur} or monograph \cite{bini2012numerical}).

\subsection{Discrete time}
\label{sec:maindisc}
In this subsection we extend the characterization of stationary processes of \cite{voutilainen2017model} to multivariate settings. Consequently, we derive quadratic equations for the corresponding model parameter matrix providing us with a natural way to define an estimator for the parameter. Finally, we pose theorems for consistency and asymptotic distribution. A strong analogue with the continuous time case $I = \mathbb{R}$ covered in \cite{voutilainen2019vector} is obtained. We start by providing some definitions.

\begin{defi}
\label{defi:incrementprocessdisc}
Let $G = (G_t)_{t\in\mathbb{Z}}$ be a $n$-dimensional stationary increment process. We define a stationary process $\Delta G = (\Delta_t G)_{t\in\mathbb{Z}}$ by
\begin{equation*}
\Delta_t G = G_t - G_{t-1}.
\end{equation*} 
\end{defi}

As in the univariate case, we define a class of stationary increment process having sub-exponentially deviating sample paths.
\begin{defi}
\label{defi:GH}
Let $H >0$ and let $G = (G_t)_{t\in\mathbb{Z}}$ be a $n$-dimensional stochastic process with stationary increments and $G_0 = 0$. If 
\begin{equation*}
\lim_{l \to -\infty} \sum_{k=l}^ 0 e^{kH} \Delta_k G
\end{equation*}
exists in probability and defines an almost surely finite random variable, we denote $G\in\mathcal{G}_H$.
\end{defi}

\begin{rem}
Lemma \ref{lemma:logarithmic} shows that existence of a logarithmic moment is sufficient for $G \in \mathcal{G}_H$ for all $H >0$. Particularly, this is the case if $G$ is square integrable. On the other hand, an example of a stationary increment process $G$ with $G_0 = 0$, but $G \notin \mathcal{G}_H$ for any $H>0$ was provided in \cite{Viitasaari-2016a}.
\end{rem}
The next theorem characterizes all multivariate stationary processes, including processes possessing long-memory.

\begin{theorem}
\label{theorem:chardisc}
Let $H > 0$ be fixed and let $X = (X_t)_{t\in\mathbb{Z}}$ be a $n$-dimensional stochastic process. Then $X$ is stationary if and only if $\lim_{t\to -\infty} e^{tH}X_t \overset{\p}{=} 0$ and
\begin{equation}
\label{characterizationdisc}
\Delta_t X = (e^{-H} -I)X_{t-1} + \Delta_t G
\end{equation}
for $G\in \mathcal{G}_H$ and $t\in\mathbb{Z}$. Moreover, the process $G\in \mathcal{G}_H$ is unique.
\end{theorem}

\begin{kor}
\label{kor:ar1}
Let $H>0$ be fixed and let $X$ be stationary. Then $X$ admits an AR$(1)$ type of representation 
\begin{equation}
\label{ar1}
X_t - \Phi X_{t-1} = \Delta_t G,
\end{equation}
where $\Phi = e^{-H}$ and $G \in \mathcal{G}_H$.
\end{kor}
By using \eqref{ar1} and the expression \eqref{mainfty} from the proof of Theorem \ref{theorem:chardisc}, it is straightforward to show that $\Delta G$ is centred and square integrable if and only if $X$ is centred and square integrable, respectively. In what follows, we assume these two attributes and write $\gamma(t) = \e X_t X_0^\top$ and $r(t) = \e (\Delta_t G)(\Delta_0 G)^\top$. Furthermore, since
\begin{equation*}
G_t = \sum_{k=1}^t \Delta_kG, \qquad t\geq 1, 
\end{equation*}
in this case also $G$ is centred and square integrable, and we denote $v(t) = \mathrm{cov}(G_t) =  \e G_t G_t^\top$. We would like to point out that centredness can be assumed without loss of generality (see Remark \ref{rem:centred}).

Under the assumptions, we obtain an expression for $\gamma(t)$ in terms of the noise process.
\begin{rem}
\label{rem:covariance}
The autocovariance function $\gamma(t)$ is given by
\begin{equation*}
\gamma(t) = e^{-tH} \sum_{k=-\infty}^t \sum_{j=-\infty}^0 e^{kH} r(k-j) e^{jH}.
\end{equation*}
Furthermore, if $G$ has independent components, we obtain
\begin{equation*}
\gamma(t) = \frac{e^{-tH}}{2} \sum_{k=-\infty}^t \sum_{j=-\infty}^0  e^{kH} \left( v(k-j+1) + v(k-j-1) - 2v(k-j)\right) e^{jH}.
\end{equation*}
\end{rem}
The following lemma writes the quadratic equations for the model parameter $\Phi = e^{-H}$ presented in \cite{voutilainen2017model} in our multivariate setting.

\begin{lemma}
\label{lemma:nonriccatidisc}
Let $H>0$ be fixed and let $X$ be stationary of the form \eqref{ar1}. Then 
\begin{equation}
\label{noncare}
r(t) = \Phi \gamma(t) \Phi - \gamma(t+1)\Phi - \Phi \gamma(t-1) +\gamma(t)
\end{equation}
for every $t\in\mathbb{Z}$.
\end{lemma}

\begin{rem}
In the proof of the lemma, we have utilized the increment process $\Delta G$ of the noise similarly as in \cite{voutilainen2017model} yielding quadratic equations for $\Phi$. However, for a general stationary $X$, \eqref{noncare} is a symmetric CARE only if $t=0$, and even in this case, existence of a unique positive semidefinite solution is not guaranteed.
\end{rem}
By applying a similar approach as in \cite{voutilainen2019vector} by considering the noise $G$ directly, we obtain a set of symmetric CAREs on which we  construct an estimator for the model parameter $\Phi = e^{-H}$. For this, we define the following matrix coefficients.

\begin{defi}
\label{defi:coefofriccatidisc}
We set
\begin{align*}
B_t &= \sum_{k=1}^t  \gamma(k-1) - \gamma(k)^\top\\
C_t &= \sum_{k=1}^ t \sum_{j=1}^ t \gamma(k-j)\\ 
D_t &= v(t) - 2\gamma(0) + \gamma(t) + \gamma(t)^\top
\end{align*}
 for every $t\in \mathbb{N}$.
\end{defi}

\begin{theorem}
\label{theorem:riccatidisc}
Let $H > 0$ be fixed and set $\Theta = I- e^{-H}$. Let $X = (X_t)_{t\in\mathbb{Z}}$ be stationary of the form \eqref{characterizationdisc}. Then the CARE
\begin{equation}
\label{riccatidisc}
B_t^\top \Theta + \Theta B_t - \Theta C_t \Theta + D_t = 0
\end{equation}
is satisfied for every $t \in\mathbb{N}$.
\end{theorem}

\begin{rem}
\label{rem:centred}
Equations \eqref{noncare} and \eqref{riccatidisc} are covariance based. Consequently, they hold also when $X$ and $G$ in Theorem \ref{theorem:chardisc} are not centred.
\end{rem}

\begin{rem}
By Lemma \ref{lemma:Phi}, the matrix $\Theta$ is positive definite. Since
\begin{equation*}
C_t = \e \sum_{k= 1}^t X_{k-1} \left( \sum_{k=1}^t X_{k-1} \right) ^\top = \mathrm{cov}\left(\sum_{k=1}^t X_{k-1}\right),
\end{equation*}
the matrix $C_t$ is positive semidefinite. Furthermore, if the smallest eigenvalue of $v(t)$ grows enough in time, $D_t$ becomes positive definite (see \cite{voutilainen2019vector}). This is the case e.g. when the noise has independent components with growing variances.
\end{rem}

We give a couple of examples on how some basic multivariate processes of ARMA type can be presented in the form \eqref{ar1}, and how to derive the corresponding noise $G$ together with its covariance function $v$.

\begin{ex}
\label{ex:ar1}
Let $X$ be a $n$-dimensional stationary AR$(1)$ type of process given by
\begin{equation*}
X_t - \phi X_{t-1} = \epsilon_t,
\end{equation*}
with $0<\phi=Q\Lambda Q^\top$, $\Vert \phi \Vert < 1$ and $\epsilon \sim iid (0, \Sigma)$. Then, we may set $H = -Q \mathrm{diag}(  \log \lambda_i) Q^\top$ giving $\Phi = \phi$. Now $\Delta G = \epsilon$ and $G_t = \sum_{k=1}^t \epsilon_k$. Furthermore, $v(t) = \sum_{k=1}^t \mathrm{cov} (\epsilon_k) = t\Sigma$ for $t\geq 1$.
\end{ex}

\begin{ex}
\label{ex:arma1q}
Let $X$ be a $n$-dimensional stationary ARMA$(1,q)$ type of process given by 
\begin{equation*}
X_t - \phi X_{t-1} = \epsilon_t + \theta_1  \epsilon_{t-1} + \hdots + \theta_q \epsilon_{t-q},
\end{equation*}
with $0 < \phi $, $\Vert \phi \Vert < 1$ and $\epsilon \sim iid (0, \Sigma)$. Similarly as above, we may set $\Phi = \phi$ and now $\Delta G$ equals to the MA$(q)$ process on the right. Consequently, for $t \geq 1$,
\begin{equation*}
G_t = \sum_{k=1}^t \epsilon_k + \theta_1 \epsilon_{k-1} + \hdots + \theta_q \epsilon_{k-q}
\end{equation*}
and
\begin{equation*}
v(t) = \sum_{i,j =0}^ q \max(0, t-|i-j|) \theta_i \Sigma \theta_j^\top ,
\end{equation*}
where $\theta_0 = I$.
\end{ex}

In \cite{voutilainen2017model} we proposed an estimation method of one-dimensional stationary processes based on equations \eqref{noncare}. In particular, we showed that the method is applicable except in some special class of stationary processes. In \cite{voutilainen2018note} we provided a comprehensive analysis of the class, and proved that it consists of highly degenerate processes. On the other hand, due to the strong dependence structure, the failure of different estimation methods is expected. Fundamentally, a stationary process $X$ belongs to the class if there exists two values $H$ and $\tilde{H}$ such that the corresponding processes $\Delta G$ and $\Delta \tilde{G}$ in \eqref{ar1} have identical autocovariance functions. Next, we state a lemma showing that these degenerate processes have a special characteristic also under the new set of equations \eqref{riccatidisc}.

\begin{lemma}
\label{lemma:cyclic}
Let $X$ be a one-dimensional stationary process and let $H > 0$ be fixed. Set $\Phi = e^{-H}$ and 
\begin{equation*}
X_t - \Phi X_{t-1} = \Delta_t G, \qquad t\in\mathbb{Z}.
\end{equation*}
If the equation
\begin{equation}
\label{quadratic1}
 \gamma(t)\Phi^2 - ( \gamma(t+1) - \gamma(t-1))\Phi + \gamma(t) - r(t) =0
\end{equation}
yields the same two solutions $\Phi, \tilde{\Phi} > 0$ for every $t\in\mathbb{Z}$, then also the equation 
\begin{equation}
\label{quadratic2}
C_t \Theta^2 - 2B_t \Theta - D_t = 0
\end{equation}
yields the same two solutions $\Theta = 1-\Phi$ and $\tilde{\Theta} = 1-\tilde{\Phi}$ for every $t\in\mathbb{N}$.
\end{lemma}
For estimation, it is desirable that \eqref{riccatidisc} admits a unique positive semidefinite solution guaranteeing convergence to the correct parameter matrix. Hence, we simply assume that $t$ is chosen in such a way that $C_t, D_t >0$, and we omit the subindex $t$ from Equation \eqref{riccatidisc}. We have justified the assumption of positive definiteness in detail in continuous time (see Subsection 2.1 and Remark 2.11 in \cite{voutilainen2019vector}). Furthermore, we assume that $v(t)$ is known and the stationary process $X$ is observed up to the time $T >t$, and the coefficient matrices $B, C$ and $D$ are estimated from these observations by replacing the autocovariances $\gamma(\cdot)$ with some estimators $\hat{\gamma}_T(\cdot)$. The coefficient estimators are denoted by $\hat{B}_T, \hat{C}_T$ and $\hat{D}_T$, and we set
\begin{equation*}
\Delta_T B =  \hat{B}_T  - B, \quad  \Delta_T C =  \hat{C}_T  - C,\quad  \Delta_T D =  \hat{D}_T  - D.
\end{equation*}
Next, we define an estimator $\hat{\Theta}_T$ for the matrix $\Theta = I - \Phi = I - e^{-H}$. The proofs of the related asymptotic results allow a certain amount of flexibility in the definition. Thus, we give a definition that probably is the most convenient from the practical point of view. Consistency and the rate of convergence of $\hat{\Theta}_T$ are inherited from autocovariance estimators $\hat{\gamma}_T(\cdot)$ of the observed stationary process. In addition, the limiting distribution is obtained as a linear function of the limiting distribution of the autocovariance estimators.

\begin{defi}
\label{defi:estimatordisc}
The estimator $\hat{\Theta}_T$ is defined as the unique positive semidefinite solution to the perturbed CARE
\begin{equation*}
\label{perCAREdisc}
\hat{\Bh}_T^\top\hat{\Theta}_T + \hat{\Theta}_T\hat{\Bh}_T - \hat{\Theta}_T\hat{\Ch}_T\hat{\Theta}_T + \hat{\Dh}_T = 0
\end{equation*}
whenever $\hat{C}_T, \hat{D}_T > 0$. Otherwise, we set $\hat{\Theta}_T = 0$.
\end{defi}

\begin{theorem}
\label{theorem:consistencydisc}
Let $C, D >0$. Assume that 
\begin{equation*}
\max_{s\in\{0,1,\hdots,t\}} \Vert \hat{\gamma}_T(s) - \gamma(s) \Vert \overset{\p}{\longrightarrow} 0.
\end{equation*}
Then
\begin{equation*}
\Vert\hat{\Theta}_T - \Theta \Vert \overset{\p}{\longrightarrow} 0,
\end{equation*}
where $\hat{\Theta}_T$ is given by Definition \ref{defi:estimatordisc}.
\end{theorem}

\begin{theorem}
\label{theorem:asymptoticsdisc}
Let $l(T)$ be a rate function. If 

\begin{equation*}
l(T)\begin{bmatrix*}
\vect (\hat{\gamma}_T(0) - \gamma(0))\\
\vect (\hat{\gamma}_T(1) - \gamma(1))\\
\vdots\\
\vect (\hat{\gamma}_T(t) - \gamma(t))\
\end{bmatrix*}
\overset{\text{law}}{\longrightarrow} Z,
\end{equation*}
where $Z$ is a $(t+1)n^2$-dimensional random vector, then:
\begin{itemize}
\item[(1)] Let $\tilde{Z}$ be the permutation of elements of $Z$ corresponding to the order of elements of
\begin{equation*}
\begin{bmatrix*}
\vect \left((\boh(0) - \gamma(0))^\top\right)\\
\vect \left((\boh(1) - \gamma(1))^\top\right)\\
\vdots\\
\vect \left((\boh(t) - \gamma(t))^\top\right)
\end{bmatrix*}.
\end{equation*}
Define a linear mapping $L_1: \mathbb{R}^ {(t+1)n^2} \to \mathbb{R}^ {3n^2}$ by 
\begin{equation*}
L_1(Z) = \begin{bmatrix*}
\sum_{k=0}^{t-1} (t-k) \begin{bmatrix*}
Z^{(kn^2+1)}\\
\vdots\\
Z^{((k+1)n^2)}
\end{bmatrix*}
+ \sum_{k=1}^{t-1} (t-k)\begin{bmatrix*}
\tilde{Z}^ {(kn^2+1)}\\
\vdots\\
\tilde{Z}^ {((k+1)n^2)}
\end{bmatrix*}\\
\sum_{k=1}^ t \begin{bmatrix*}
Z^{((k-1)n^2+1)} - \tilde{Z}^{(kn^2+1)}\\
\vdots\\
Z^{(kn^2)} - \tilde{Z}^ {((k+1)n^2)}
\end{bmatrix*}\\

 \begin{bmatrix*}
Z^{(tn^2+1)} + \tilde{Z}^ {(tn^2+1)}\\
\vdots\\
Z^{((t+1)n^2)} + \tilde{Z}^ {((t+1)n^2)}
\end{bmatrix*}
-2\begin{bmatrix*}
Z^{(1)}\\
\vdots\\
Z^ {(n^2)}
\end{bmatrix*}
\end{bmatrix*},
\end{equation*}
where $\sum_1^0$ is an empty sum. Then 
\begin{equation*}
l(T) \vect (\Delta_T C, \Delta_T B, \Delta_T D) \overset{\text{law}}{\longrightarrow} L_1(Z).
\end{equation*}
\item[(2)] If $D, C > 0$ and $\hat{\Theta}_T$ is given by Definition \ref{defi:estimatordisc}, then
\begin{equation*}
l(T) \vect(\hat{\Theta}_T-\Theta) \overset{\text{law}}{\longrightarrow} L_2(L_1(Z)),
\end{equation*}
where $L_2 : \mathbb{R}^ {3n^2} \to \mathbb{R}^ {n^2}$ is a linear mapping expressible in terms of $\Theta, t$ and $r$.
\end{itemize}
\end{theorem}

%

\subsection{Continuous time}
\label{sec:maincont}
We have collected the main results (Theorems \ref{theorem:charcont}, \ref{theorem:riccaticont} and \ref{theorem:fclt}) of \cite{voutilainen2019vector} considering continuous time stationary processes into this subsection. In addition, in order to complete the analogue between discrete and continuous time, we derive quadratic equations for the model parameter by using the increments of the noise process (Lemma \ref{lemma:nonriccaticont}). We assume that the processes have continuous paths almost surely and hence, the related stochastic integrals can be interpreted as pathwise Riemann-Stieltjes integrals. Again, we start by defining the class $\mathcal{G}_H$ of stationary increment processes for $H > 0$.

\begin{defi}
\label{defi:GH2}
Let $H >0$ and let $G = (G_t)_{t\in \mathbb{R}}$ be a $n$-dimensional stochastic process with stationary increments and $G_0 = 0$. If 
\begin{equation*}
\lim_{s\to -\infty} \int_s^0 e^{Hu} dG_u
\end{equation*}
exists in probability and defines an almost surely finite random variable, we denote $G \in \mathcal{G}_H$.
\end{defi}
As in discrete time, it can be shown that existence of some logarithmic moments ensure that $G \in \mathcal{G}_H$ for all $H > 0$. In particular, square integrability of $G$ suffices, which is the case in our second moment based estimation method.

The next theorem is the continuous time counterpart of Theorem \ref{theorem:chardisc} showing that all stationary processes are characterized by the Langevin equation, whereas in discrete time, the characterization was given by an AR$(1)$ type of equation.

\begin{theorem}
\label{theorem:charcont}
Let $H > 0$ be fixed and let $X = (X_t)_{t\in\mathbb{R}}$ be a $n$-dimensional stochastic process. Then $X$ is stationary if and only if
\begin{equation*}
\label{initial}
 X_0 = \int_{-\infty}^0 e^{Hu} dG_u
\end{equation*}
and
\begin{equation}
\label{langevin}
dX_t = -HX_t dt + dG_t, \quad
\end{equation}
for $G \in\mathcal{G}_H$ and $t\in\mathbb{R}$. Moreover, the process $G \in \mathcal{G}_H$ is unique.
\end{theorem}

\begin{kor}
From Theorem \ref{theorem:charcont} it follows that $X$ is the unique stationary solution
\begin{equation}
\label{stationarysolution}
X_t = e^{-Ht} \int_{-\infty}^t e^{Hu} dG_u
\end{equation}
to \eqref{langevin}.
\end{kor}
In order to apply Theorem \ref{theorem:charcont} in estimation, we pose the assumption
\begin{equation*}
\sup_{s\in[0,1]}\e \Vert G_s \Vert^2 < \infty.
\end{equation*}
This guarantees that $G \in \mathcal{G}_H$ for all $H >0$, and square integrability of $X$ and $G$. On the other hand, if $X$ is square integrable, then $G$ is also. In addition and without loss of generality, we assume that the processes are centred. Again, we write $\gamma(t) = \e X_t X_0^\top$ and $v(t) = \e G_t G_t^\top$. Now, the autocovariance function of the following stationary process is well-defined.
\begin{defi}
\label{defi:incrementprocess}
Let $G = (G_t)_{t\in\mathbb{R}}$ be a centred square integrable stationary increment process and let $\delta > 0$. We define a stationary process $\Delta^\delta G = (\Delta_t^\delta G)_{t\in\mathbb{R}}$ by
\begin{equation*}
\Delta_t^\delta G = G_t - G_{t-\delta}
\end{equation*} 
and the corresponding autocovariance function $r_\delta$ by 
\begin{equation*}
r_\delta(t) =\e (\Delta_t^\delta G)(\Delta_0^\delta G)^\top.
\end{equation*}
\end{defi}
As in discrete time (Lemma \ref{lemma:nonriccatidisc}), we obtain quadratic equations for the model parameter $H$ in terms of $r_\delta$.
\begin{lemma}
\label{lemma:nonriccaticont}
Let $H>0$ be fixed and let $X$ be of the form \eqref{stationarysolution}. Then 
\begin{scriptsize}
\begin{equation}
\label{noncarecont}
\begin{split}
r_\delta(t) = &2\gamma(t) - \gamma(t+\delta) - \gamma(t-\delta) + \left(\int_t^{t+\delta} \gamma(s) ds - \int_{t-\delta}^ t \gamma(s) ds \right) H + H\left(\int_{t-\delta}^ t \gamma(s) ds - \int_t^{t+\delta} \gamma(s) ds \right)\\
& + H \left(\int_{t-\delta}^ t (s-t+\delta)\gamma(s) ds + \int_t^{t+\delta} (t-s+\delta) \gamma(s) ds \right) H
\end{split}
\end{equation}
\end{scriptsize}
for every $t\in\mathbb{R}$.
\end{lemma}

\begin{rem}
The advantage of the equations above is that we have to consider $\gamma(s)$ only for $s\in[t-\delta, t+\delta]$, but as in the discrete case, for a general stationary $X$ we obtain a symmetric CARE only when $t=0$. In addition, similarly as above, we could set in discrete time $\Delta_t^k G \coloneqq G_t - G_{t-k}, k\in\mathbb{N}$. However, this would lead to more complicated equations in Lemma \ref{lemma:nonriccatidisc}.
\end{rem}
A significant difference compared to the discrete time equations \eqref{noncare} occurs in the univariate case. Namely, the first order term with respect to $H$ vanishes.

\begin{kor}
\label{kor:1Dcont}
The univariate case yields
\begin{footnotesize}
\begin{equation*}
r_\delta(t) = 2\gamma(t) - \gamma(t+\delta) - \gamma(t-\delta) + H^2 
 \left(\int_{t-\delta}^ t (s-t+\delta)\gamma(s) ds + \int_t^{t+\delta} (t-s+\delta) \gamma(s) ds \right)
\end{equation*}
\end{footnotesize}
for every $t\in\mathbb{R}$.
\end{kor}
One could potentially base a univariate estimation method on the above equations without the concern of existence of a unique positive solution. However, since we wish to treat also multivariate settings, we present the most central results of \cite{voutilainen2019vector} that are obtained from Theorem \ref{theorem:charcont} by considering the noise $G$ directly. First, we define matrix coefficients corresponding to Definition \ref{defi:coefofriccatidisc}. Consequently, we write symmetric CAREs for the parameter $H$ that are similar to the CAREs \eqref{riccatidisc} for the discrete time parameter $\Theta$.

\begin{defi}
\label{defi:coefofriccaticont}
We set 
\begin{align*}
B_t &= \int_0^t  \gamma(s) - \gamma(s)^\top  ds \\
C_t &= \int_0^t \int_0^t \gamma(s-u) du ds \\
D_t &= v(t) - 2\gamma(0) + \gamma(t) + \gamma(t)^\top
\end{align*}
for every $t \geq0$.
\end{defi}

\begin{theorem}
\label{theorem:riccaticont}
Let $H > 0$ be fixed and let $X=(X_t)_{t\in\mathbb{R}}$ be stationary of the form \eqref{stationarysolution}. Then the CARE
\begin{equation}
\label{riccaticont}
B_t^\top H + H B_t - H C_t H + D_t = 0
\end{equation}
 is satisfied for every $t \geq 0$.
\end{theorem}

\begin{rem}
As in discrete time, the equations \eqref{noncarecont} and \eqref{riccaticont} are covariance based and hence, they hold also for a non-centred stationary $X$.
\end{rem}

\begin{rem}
Contrary to the discrete time equations \eqref{riccatidisc}, the first order term vanishes in the univariate setting as in \eqref{noncarecont}.
\end{rem}
Again, we assume that $t$ is chosen in such a way that $C_t, D_t >0$ ensuring the existence of a unique positive semidefinite solution. We have discussed this assumption in detail in \cite{voutilainen2019vector}. 
We define an estimator $\hat{H}_T$ for the model parameter matrix $H$ identically to the discrete time by replacing the autocovariances $\gamma(\cdot)$ in the matrix coefficients with their estimators $\hat{\gamma}_T(\cdot)$. The below given definition differs slightly from the definition in \cite{voutilainen2019vector}, but the same asymptotic results still apply.

\begin{defi}
\label{defi:estimatorcont}
The estimator $\hat{H}_T$ is defined as the unique positive semidefinite solution to the perturbed CARE
\begin{equation*}
\label{perCARE}
\hat{\Bh}_T^\top\hat{H}_T + \hat{H}_T\hat{\Bh}_T - \hat{H}_T\hat{\Ch}_T\hat{H}_T + \hat{\Dh}_T = 0
\end{equation*}
whenever $\hat{C}_T, \hat{D}_T > 0$. Otherwise, we set $\hat{H}_T = 0$.
\end{defi}
As in discrete time, asymptotic properties of $\hat{H}_T$ are inherited from the autocovariance estimators. However, due to the continuous time setting, instead of pointwise convergence, we have to consider functional form of convergence of $\hat{\gamma}_T(\cdot)$. In \cite{voutilainen2019vector} we have provided sufficient conditions in the case of Gaussian noise $G$ with independent components, under which the assumptions of the following theorems are satisfied. In particular, the results are valid for fractional Brownian motion that is widely applied in the field of mathematical finance.
\begin{theorem}
\label{theo:pertur}
Let $C, D >0$. Assume that 
\begin{equation*}
\label{covarianceinP}
\sup_{s\in[0,t]} \Vert\hat{\gamma}_{T}(s)- \gamma(s)\Vert \overset{\p}{\longrightarrow} 0.
\end{equation*}
Then
\begin{equation*}
\label{solutioninP}
\Vert\hat{\bh}_T - \bh\Vert \overset{\p}{\longrightarrow} 0,
\end{equation*}
where $\hat{\bh}_T$ is given by Definition \ref{defi:estimatorcont}.
\end{theorem}

\begin{theorem}
\label{theorem:fclt}
Let $Y = (Y_s)_{s\in[0,t]}$ be an $n^2$-dimensional stochastic process with continuous paths almost surely and let $l(T)$ be a rate function. If
\begin{equation*}
\label{functionalcovariance}
l(T)\vect (\hat{\gamma}_T(s)- \gamma(s)) \overset{\text{law}}{\longrightarrow} Y_s
\end{equation*}
in the uniform topology of continuous functions, then: \begin{itemize}
\item[(1)] Let $\tilde{Y}_s$ be the permutation of elements of $Y_s$ that corresponds to the order of elements of $\vect\left((\boh(s) -\gamma(s))^\top\right)$. Then
\begin{equation*}
l(T) \vect(\Delta_T \Ch, \Delta_T \Bh, \Delta_T \Dh) \overset{\text{law}}{\longrightarrow}  \begin{bmatrix*}
\int_0^ t (t-s)(Y_s + \tilde{Y}_s) ds\\
\int_0^t \left(Y_s -\tilde{Y}_s\right) ds\\
 Y_t +\tilde{Y}_t - 2Y_0
\end{bmatrix*}
\eqqcolon L_1(Y).
\end{equation*}
\item[(2)] If $\Ch, \Dh > 0$ and $\hat{\bh}_T$ is given by Definition \ref{defi:estimatorcont}, then
\begin{equation*}
l(T)  \vect(\hat{\bh}_T - \bh) \overset{\text{law}}{\longrightarrow} L_2(L_1(Y)),
\end{equation*}
where $L_2: \mathbb{R}^{3n^2} \rightarrow \mathbb{R}^{n^2}$ is a linear mapping expressible in terms of $\bh$, $t$ and the covariance function of $G$.
\end{itemize}
\end{theorem}

\section{Proofs}
\label{sec:proofs}
In the following, we denote the smallest eigenvalue of $H >0$ by $\lambda_{min}$. Consequently $\Vert e^{kH} \Vert = \Vert Q \mathrm{diag}(e^{\lambda_ik}) Q^\top \Vert = e^{\lambda_{min} k}$ for a negative $k$.
\subsection{Discrete time}
\label{sub:proofsdiscrete}
The proof of the next lemma follows the lines of the proof of Theorem 2.2. in \cite{Viitasaari-2016a} that concerns the one-dimensional continuous time case. However, in our setting, we obtain a weaker sufficient condition for $G \in\mathcal{G}_H$ for all $H > 0$.

\begin{lemma}
\label{lemma:logarithmic}
Let $G = (G_t)_{t\in\mathbb{Z}}$ be a stationary increment process with $G_0=0$. Assume that
\begin{equation*}
\e \left | \log \Vert G_1 \Vert \mathbbm{1}_{\{\Vert G_1 \Vert > 1\}} \right|^{1 + \delta} < \infty
\end{equation*}
for some $\delta >0$. Then $G \in \mathcal{G}_H$ for all $H >0$.
\begin{proof}
Let $H >0$. We apply the Borel-Cantelli lemma together with Markov's inequality to show that $\Vert e^{kH} \Delta_k G \Vert \to 0$ almost surely as $k\to -\infty$. Let $\epsilon > 0$ be fixed below.

\begin{equation*}
\begin{split}
\p \left(\Vert e^{kH} \Delta_k G \Vert > \epsilon\right) &\leq \p \left(e^{\lambda_{min} k} \Vert \Delta_k G \Vert > \epsilon \right) = \p \left(e^{\lambda_{min} k} \Vert G_1 \Vert > \epsilon \right)\\
 &= \p \left( \Vert G_1 \Vert > \frac{\epsilon}{e^{\lambda_{min} k}} \right)= \p \left( \log \Vert G_1 \Vert > \log \epsilon - \lambda_{min} k \right),
\end{split}
\end{equation*}
since $\Delta G$ is stationary and $G_0 =0$. Furthermore, $k(\frac{\log \epsilon}{k} - \lambda_{min}) \geq -Ck$ for some $C >0$ and $k \leq k_\epsilon$. Thus, for $k \leq k_\epsilon$,
\begin{equation*}
\begin{split}
\p \left(\Vert e^{kH} \Delta_k G \Vert > \epsilon\right) &\leq \p\left( \log \Vert G_1 \Vert \geq -Ck\right) = \p\left( \log \Vert G_1 \Vert \mathbbm{1}_{\{\Vert G_1 \Vert > 1\}}  \geq -Ck\right)\\
&\leq \frac{\e \left | \log \Vert G_1 \Vert \mathbbm{1}_{\{\Vert G_1 \Vert > 1\}} \right|^{1 + \delta}}{(-Ck)^{1+\delta}} \leq c \frac{1}{(-k)^{1+\delta}}
\end{split}
\end{equation*}
giving the wanted result. We conclude the proof by noting that
\begin{footnotesize}
\begin{equation*}
\sum_{k=-\infty}^0 \Vert e^{kH} \Delta_k G\Vert \leq \sum_{k=-\infty}^0  \Vert e^{\frac{1}{2} k H} \Vert \Vert e^{\frac{1}{2} k H} \Delta_k G \Vert \leq \sup_{k} \Vert e^{\frac{1}{2} kH} \Delta_k G \Vert \sum_{k=-\infty}^0   e^{\frac{1}{2} \lambda_{min} k}  < \infty
\end{equation*}
\end{footnotesize}
almost surely.
\end{proof}
\end{lemma}

We next extend the concept of self-similarity to discrete time multivariate processes.
\begin{defi}
\label{defi:selfsimilarity}
Let $H > 0$ and let $Y = (Y_{e^t})_{t\in\mathbb{Z}}$ be a $n$-dimensional stochastic process. Then $Y$ is $H$-self-similar if 
\begin{equation*}
(Y_{e^{t+s}})_{t\in\mathbb{Z}} \overset{\text{law}}{=} (e^{sH}Y_{e^t})_{t\in\mathbb{Z}}
\end{equation*}
for every $s\in\mathbb{Z}$.
\end{defi}

The following transform and the corresponding theorem giving one-to-one correspondence between self-similar and stationary processes were originally introduced by Lamperti in the univariate continuous time setting (\cite{Lamperti-1962}).
\begin{defi}
\label{defi:lamberti}
Let $H >0$, and let $X = (X_t)_{t\in\mathbb{Z}}$ and $Y = (Y_{e^t})_{t\in\mathbb{Z}}$ be $n$-dimensional stochastic processes. We define
\begin{equation*}
(\mathcal{L}_H X)_{e^t} = e^{tH} X_t
\end{equation*}
and
\begin{equation*}
(\mathcal{L}_H^{-1} Y)_t  = e^{-tH} Y_{e^t}.
\end{equation*}
\end{defi}

\begin{theorem}
\label{theorem:lamperti}
The operator $\mathcal{L}_H$ together with its inverse $\mathcal{L}_H^ {-1}$ define a bijection between $n$-dimensional stationary processes and $n$-dimensional $H$-self-similar processes.

\begin{proof}
First, let $X$ be stationary and set $Z_{e^t}= (\mathcal{L}_H X)_{e^t}$. Then 
\begin{equation*}
\begin{bmatrix*}
Z_{e^{t_1+s}}\\
Z_{e^{t_2+s}}\\
\vdots\\
Z_{e^{t_m+s}}
\end{bmatrix*} = \begin{bmatrix*}
e^{(t_1+s)H} X_{t_1+s}\\
e^{(t_2+s)H} X_{t_2+s}\\
\vdots\\
e^{(t_m+s)H} X_{t_m+s}
\end{bmatrix*}  \overset{\text{law}}{=} \begin{bmatrix*}
e^{(t_1+s)H} X_{t_1}\\
e^{(t_2+s)H} X_{t_2}\\
\vdots\\
e^{(t_m+s)H} X_{t_m}
\end{bmatrix*} = \begin{bmatrix*}
e^{sH} Z_{e^{t_1}}\\
e^{sH} Z_{e^{t_1}}\\
\vdots\\
e^{sH} Z_{e^{t_m}}
\end{bmatrix*}
\end{equation*}
for every $m\in\mathbb{N}$, $t\in\mathbb{Z}^m$ and $s\in\mathbb{Z}$. Hence, $Z$ is $H$-self-similar.

Now, let $Y$ be $H$-self-similar and set $Z_t = (\mathcal{L}_H^{-1} Y)_t$. Then
\begin{equation*}
\begin{bmatrix*}
Z_{t_1 +s}\\
Z_{t_2 +s}\\
\vdots\\
Z_{t_m +s}
\end{bmatrix*} = \begin{bmatrix*}
e^{-(t_1+s)H} Y_{e^{t_1 + s}}\\
e^{-(t_2+s)H} Y_{e^{t_2 + s}}\\
\vdots\\
e^{-(t_m+s)H} Y_{e^{t_m + s}}
\end{bmatrix*}  \overset{\text{law}}{=} \begin{bmatrix*}
e^{-t_1H}Y_{e^{t_1}}\\
e^{-t_2H}Y_{e^{t_2}}\\
\vdots\\
e^{-t_mH}Y_{e^{t_m}}
\end{bmatrix*} = \begin{bmatrix*}
Z_{t_1}\\
Z_{t_2}\\
\vdots\\
Z_{t_m}
\end{bmatrix*}
\end{equation*}
for every $m\in\mathbb{N}$, $t\in\mathbb{Z}^m$ and $s\in\mathbb{Z}$. Hence, $Z$ is stationary completing the proof.
\end{proof}
\end{theorem}
Before the proof of Theorem \ref{theorem:chardisc} we state an auxiliary lemma.

\begin{lemma}
\label{lemma:YtoG}
Let $H>0$ and let $(Y_{e^t})_{t\in\mathbb{Z}}$ be a $n$-dimensional $H$-self-similar process. We define a process $G = (G_t)_{t\in\mathbb{Z}}$ by

\begin{equation*}
G_t = \left\{
\begin{array}{ll}
\sum_{k=1}^t e^{-kH} \Delta_k Y_{e^k}, \quad &t\geq 1 \\
0, \quad &t=0 \\
-\sum_{k=t+1}^0 e^{-kH} \Delta_k Y_{e^k}, \quad &t\leq -1.
\end{array}
\right.
\end{equation*}
Then $G \in \mathcal{G}_H$.
\begin{proof}
It is straightforward to verify that $\Delta_t G = e^{-tH} \Delta_t Y_{e^t}$ for every $t\in\mathbb{Z}$. In addition
\begin{equation*}
\lim_{l \to -\infty} \sum_{k=l}^ 0 e^{kH} \Delta_k G = \lim_{l \to -\infty} \sum_{k=l}^ 0 \Delta_k Y_{e^ k} = Y_1 - \lim_{l\to -\infty} Y_{e^ {l-1}}, 
\end{equation*}
where by self-similarity of $Y$
\begin{equation*}
\p (\Vert Y_{e^{l-1}}\Vert \geq \epsilon) = \p(\Vert e^{H(l-1)} Y_1 \Vert\geq \epsilon) \leq \p(e^{\lambda_{min}(l-1)} \Vert Y_1\Vert \geq \epsilon) \to 0. 
\end{equation*}
Hence, we set
\begin{equation*}
\lim_{l \to -\infty} \sum_{k=l}^ 0 e^{kH} \Delta_k G = Y_1.
\end{equation*}
\end{proof}
\end{lemma}

\begin{proof}[Proof of Theorem \ref{theorem:chardisc}]
Assume that $\lim_{t\to -\infty} e^{t \bh} X_t \overset{\p}{=} 0$ and \eqref{characterizationdisc} holds for $G\in\mathcal{G}_\bh$. Then, by using \eqref{characterizationdisc} repeatedly
\begin{footnotesize}
\begin{equation*}
\begin{split}
X_t &= e^{-\bh} X_{t-1} + \Delta_t G = e^{-(n+1)\bh} X_{t-n-1} + \sum_{j=0}^n e^{-j \bh} \Delta_{t-j} G  \\
&=  e^{-(n+1)\bh} X_{t-n-1} + e^{-tH}\sum_{k = t-n}^t e^{k \bh} \Delta_k G =  e^{-t \bh} \left( e^{(t-n-1)\bh} X_{t-n-1}+ \sum_{k = t-n}^t e^{k \bh} \Delta_k G \right)
\end{split}
\end{equation*}
\end{footnotesize}
for every $n\in\mathbb{N}$. Since, as $n\to\infty$, the limit of the sum above is well-defined, and $\lim_{n\to \infty} e^{(t-n-1)\bh}X_{t-n-1} \overset{\p}{=} 0$, we obtain that
\begin{equation}
\label{mainfty}
X_t = e^{-t\bh} \sum_{k = -\infty} ^t e^{k\bh} \Delta_k G.
\end{equation}
Let $m \in\mathbb{N}, t\in\mathbb{Z}^m$ and $s\in\mathbb{Z}$. Then, by stationary increments of $G$, we have
\begin{equation*}
\begin{bmatrix*}
e^{-(t_1+s)H} \sum_{j=-M} ^{t_1} e^{(j+s)H} \Delta_{j+s} G\\
\vdots\\
e^{-(t_m+s)H} \sum_{j=-M} ^{t_m} e^{(j+s)H} \Delta_{j+s} G
\end{bmatrix*} \overset{\text{law}}{=} \begin{bmatrix*}
e^{-t_1H} \sum_{j=-M} ^{t_1} e^{jH} \Delta_j G\\
\vdots\\
e^{-t_mH} \sum_{j=-M} ^{t_m} e^{jH} \Delta_{j} G
\end{bmatrix*}
\end{equation*}
for every $-M < \min\{t_i\}$. Since the random vectors above converge in probability as $M\to\infty$, we obtain that
\begin{footnotesize}
\begin{equation*}
\begin{bmatrix*}
X_{t_1+ s}\\
\vdots\\
X_{t_m+s}
\end{bmatrix*} = \begin{bmatrix*}
e^{-(t_1+s)H} \sum_{j=-\infty} ^{t_1} e^{(j+s)H} \Delta_{j+s} G\\
\vdots\\
e^{-(t_m+s)H} \sum_{j=-\infty} ^{t_m} e^{(j+s)H} \Delta_{j+s} G
\end{bmatrix*} \overset{\text{law}}{=} \begin{bmatrix*}
e^{-t_1H} \sum_{j=-\infty} ^{t_1} e^{jH} \Delta_j G\\
\vdots\\
e^{-t_mH} \sum_{j=-\infty} ^{t_m} e^{jH} \Delta_{j} G
\end{bmatrix*} = \begin{bmatrix*}
X_{t_1}\\
\vdots\\
X_{t_m}
\end{bmatrix*}
\end{equation*}
\end{footnotesize}
and hence, $X$ is stationary.

Next, assume that $X$ is stationary. Then, by Theorem \ref{theorem:lamperti} there exists a $H$-self-similar $Y$ such that
\begin{equation*}
\Delta_t X = e^{-tH} Y_{e^t} - e^{-(t-1)H} Y_{e^{t-1}} = (e^{-H} - I)X_{t-1} + e^{-tH} \Delta_t Y_{e^t}.
\end{equation*}
Defining $G$ as in Lemma \ref{lemma:YtoG} completes the proof of the other direction. 

To prove uniqueness, we use \eqref{mainfty}. Assume that, for $G, \tilde{G} \in \mathcal{G}_H$,
\begin{equation*}
e^{tH} X_t = \sum_{k=-\infty}^t e^{kH} \Delta_k G = \sum_{k=-\infty}^t e^{kH} \Delta_k \tilde{G}
\end{equation*}
for every $t\in\mathbb{Z}$. Then 
\begin{equation*}
e^{tH}X_t - e^{(t-1)H}X_{t-1} = e^{tH} \Delta_t G = e^{tH} \Delta_t \tilde{G}.
\end{equation*}
Since $e^{tH}$ is invertible and both processes start from zero, we conclude that $G = \tilde{G}$.
\end{proof}

\begin{lemma}
\label{lemma:Phi}
The matrix $\Theta = I- e^{-H}$ is positive definite.
\begin{proof}
Let $a$ be a real vector of length $n$, and let $H = Q\Lambda Q^\top$ be an eigendecomposition of $H$. Then
\begin{equation*}
a^\top (I -e^ {-H}) a = \Vert a\Vert^2 - a^\top e^{-H} a,
\end{equation*}
where 
\begin{equation*}
 |a^\top e^{-H} a| \leq  \Vert a\Vert^2 e^{-\lambda_{min}} < \Vert a\Vert^2
\end{equation*}
completing the proof.
\end{proof}
\end{lemma}

\begin{proof}[Proof of Lemma \ref{lemma:nonriccatidisc}]
We have that
\begin{equation*}
\Delta_t G (\Delta_0 G)^\top = (X_t - \Phi X_{t-1})(X_0^\top - X_{-1}^\top \Phi).
\end{equation*}
Taking expectations yields
\begin{equation*}
r(t) = \Phi \gamma(t) \Phi - \gamma(t+1)\Phi - \Phi \gamma(t-1) +\gamma(t).
\end{equation*}
\end{proof}

\begin{proof}[Proof of Theorem \ref{theorem:riccatidisc}]
Let 
\begin{equation*}
\Delta_t X=  - \Theta X_{t-1} + \Delta_t G.
\end{equation*}
Then for $t\in\mathbb{N}$ we have 
\begin{equation*}
G_t = \sum_{k=1}^t \Delta_k G = \sum_{k=1}^ t \Delta_k X + \Theta \sum_{k=1}^ t X_{k-1} = X_t - X_0 + \Theta \sum_{k=1}^ t X_{k-1}.
\end{equation*}
Hence
\begin{tiny}
\begin{equation*}
\mathrm{cov}(G_t) = \mathrm{cov}(X_t-X_0) + \e\left[(X_t-X_0)\sum_{k=1}^t X_{k-1}^\top \right] \Theta + \Theta \e\left[\sum_{k=1}^ t X_{k-1} (X_t-X_0)^\top \right] + \Theta\e\left[ \sum_{k=1}^ t X_{k-1} \left(\sum_{k=1}^ t X_{k-1} \right)^\top \right] \Theta
\end{equation*}
\end{tiny}
giving \eqref{riccatidisc} since
\begin{equation*}
\e\left[(X_t-X_0)\sum_{k=1}^t X_{k-1}^\top \right] = \sum_{k=1}^ t \gamma(t-k+1) - \gamma(-k+1) = \sum_{k=1}^t \gamma(k) - \gamma(k-1)^\top
\end{equation*}
and
\begin{footnotesize}
\begin{equation*}
\mathrm{cov}(X_t -X_0) = \e\left[(X_t-X_0)(X_t-X_0)^ \top\right] = 2\gamma(0) - \gamma(t) - \gamma(-t) = 2\gamma(0) - \gamma(t) - \gamma(t)^\top
\end{equation*}
\end{footnotesize}
\end{proof}

\begin{proof}[Proof of Lemma \ref{lemma:cyclic}]
Assume that $\tilde{\Phi} = e^{-\tilde{H}}$ satisfies \eqref{quadratic1} for every $t\in\mathbb{Z}$ and set
\begin{equation*}
X_t - \tilde{\Phi} X_{t-1} = \Delta_t \tilde{G}, \qquad t\in\mathbb{Z}.
\end{equation*}
Consequently
\begin{equation*}
\tilde{r}(t) =  \gamma(t)\tilde{\Phi}^2 -  ( \gamma(t+1) - \gamma(t-1))\tilde{\Phi} + \gamma(t) = r(t), \qquad t\in\mathbb{Z},
\end{equation*}
where $\tilde{r}(t)$ is the autocovariance function of $(\Delta_t \tilde{G})_{t\in\mathbb{Z}}$.
Now, since $G_0 = \tilde{G}_0 = 0$, we obtain that
\begin{footnotesize}
\begin{equation*}
\mathrm{var}(G_t) = \mathrm{var}\left(\sum_{k=1}^t \Delta_k G\right) = \sum_{k,j=1}^t \mathrm{cov} ( \Delta_k G, \Delta_j G) = \sum_{k,j=1}^t r(k-j) = \mathrm{var}\left(\sum_{k=1}^ t \Delta_k \tilde{G} \right) = \mathrm{var}(\tilde{G}_t)
\end{equation*}
\end{footnotesize}
for all $t\in\mathbb{N}$. Hence, both $\Theta$ and $\tilde{\Theta}$ are solutions to \eqref{quadratic2}. 
\end{proof}
In order to show that $\hat{\Theta}_T$is consistent, we simply need to find suitable bounds for $\Delta_T B, \Delta_T C$ and $\Delta_T D$ in terms of the autocovariance estimators. After that, the same strategy as in \cite{voutilainen2019vector} can be applied.

\begin{lemma}
\label{lemma:bounds}
Set
\begin{equation*}
M_{t,T} = \max_{s\in\{0,1,\hdots,t\}} \Vert \hat{\gamma}_{T}(s) - \gamma(s) \Vert.
\end{equation*}
Then the coefficients of the perturbed CARE satisfy
\begin{align*}
\Vert\Delta_T D \Vert&\leq 4 M_{t,T}\\
\Vert\Delta_T C\Vert &\leq t^2 M_{t,T}\\
\Vert \Delta_T B \Vert &\leq 2t M_{t,T},
\end {align*}
\begin{proof}
First, we recall first that 
\begin{equation*}
\Vert \hat{\gamma}_T(-s) - \gamma(-s) \Vert = \Vert \hat{\gamma}_T(s)^\top - \gamma(s)^\top \Vert = \Vert \hat{\gamma}_T(s) - \gamma(s) \Vert. 
\end{equation*}
Now, since $v(t)$ is known,
\begin{equation*}
\begin{split}
\Vert\Delta_T D \Vert \leq 2 \Vert\boh(0) - \gamma(0)\Vert + \Vert\boh(t) -\gamma(t)  \Vert + \Vert  \boh(t)^ \top- \gamma(t)^\top\Vert \leq 4 M_{t,T}.
\end{split}
\end{equation*}
Moreover
\begin{equation*}
\Vert\Delta_T C\Vert \leq \sum_{k=1}^ t \sum_{j=1}^ t \Vert \hat{\gamma}_T(k-j)- \gamma(k-j)\Vert  \leq t^2 M_{t,T}. 
\end{equation*}
Finally
\begin{equation*}
\begin{split}
 \Vert \Delta_T B\Vert &\leq \sum_{k=1}^t \Vert \boh(k-1) -  \gamma(k-1)\Vert + \Vert \gamma(k)^\top  - \boh(k)^\top \Vert \leq 2t M_{t,T}.
\end{split}
\end{equation*}
\end{proof}
\end{lemma}

\begin{proof}[Proof of Theorem \ref{theorem:consistencydisc}]
The result follows by replacing $\sup_{s\in [0,t]} \Vert \boh(s) - \gamma(s) \Vert$ with $M_{t,T}$ in Corollary 3.14 and in the proof of Theorem 2.9 of \cite{voutilainen2019vector}. The details are left to the reader.
\end{proof}

\begin{proof}[Proof of Theorem \ref{theorem:asymptoticsdisc}]
For the first part of the theorem, we notice that
\begin{equation*}
\begin{split}
C_t &= \sum_{k=1}^ t \sum_{j=1}^ t \gamma(k-j) = \sum_{k=1}^t \sum_{l=k-t}^{k-1} \gamma(l) = \sum_{l=1-t}^ {-1} \sum_{k=1}^ {t+l} \gamma(l) + \sum_{l=0}^ {t-1} \sum_{k=l+1}^t \gamma(l)\\
&= \sum_{1-t}^{-1} (t+l) \gamma(l) + \sum_{l=0}^ {t-1} (t-l) \gamma(l) =  \sum_{l=0}^ {t-1} (t-l) \gamma(l) + \sum_{l=1}^ {t-1} (t-l) \gamma(l)^\top,
\end{split}
\end{equation*}
where $\sum_0^{-1}$ and $\sum_1^0$ are interpreted as empty sums. Now we have that
\begin{align*}
\Delta_T C &= \sum_{k=0}^ {t-1} (t-k) (\hat{\gamma}_T(k) - \gamma(k)) + \sum_{k=1}^ {t-1} (t-k)\left(\hat{\gamma}_T(k)^\top - \gamma(k)^\top\right)\\
\Delta_T B &= \sum_{k=1}^t \hat{\gamma}_T(k-1) - \gamma(k-1) - \hat{\gamma}_T(k)^\top + \gamma(k)^\top\\
\Delta_T D & =  2( \gamma(0)-\hat{\gamma}_T(0)) + \hat{\gamma}_T(t) - \gamma(t)   + \hat{\gamma}_T(t)^\top - \gamma(t)^\top 
\end{align*}
and furthermore
\begin{scriptsize}
\begin{equation*}
\begin{split}
l(T) \vect (\Delta_T C, \Delta_T B, \Delta_T D) &= 
l(T)\begin{bmatrix*}
\sum_{k=0}^{t-1} (t-k) \vect (\boh(k) - \gamma(k)) + \sum_{k=1}^{t-1} (t-k) \vect \left((\boh(k)- \gamma(k))^\top \right)\\
\sum_{k=1}^t \vect (\boh(k-1) - \gamma(k-1)) - \vect\left((\boh(k) - \gamma(k))^\top\right)\\
-2\vect(\boh(0) - \gamma(0)) + \vect(\boh(t) - \gamma(t)) + \vect\left((\boh(t)- \gamma(t))^\top \right)
\end{bmatrix*}\\
 &= 
L_1\left( l(T) \begin{bmatrix*} \vect (\hat{\gamma}_T(0) - \gamma(0))\\
\vect (\hat{\gamma}_T(1) - \gamma(1))\\
\vdots\\
\vect (\hat{\gamma}_T(t) - \gamma(t)) \end{bmatrix*} \right) \overset{\text{law}}{\longrightarrow} L_1(Z)
\end{split}
\end{equation*}
\end{scriptsize}
by the continuous mapping theorem. For the second part of the theorem, the proof of the continuous time case of \cite{voutilainen2019vector} can be applied just by replacing $\sup_{s\in [0,t]} \Vert \boh(s) - \gamma(s) \Vert$ with $M_{t,T}$ in the definition of the set $A_T$.
\end{proof}

\subsection{Continuous time}
We only provide the proof of Lemma \ref{lemma:nonriccaticont}, while the other proofs can be found from \cite{voutilainen2019vector}.
\label{sub:proofscontinuous}
\begin{proof}[Proof of Lemma \ref{lemma:nonriccaticont}]
Integrating \eqref{langevin} from $0$ to $t$ gives 
\begin{equation*}
G_t = X_t - X_0 + H\int_0^t X_s ds.
\end{equation*}
Hence
\begin{equation*}
\Delta_t^\delta G (\Delta_0^\delta G)^\top = \left(X_t - X_{t-\delta} + H\int_{t-\delta}^ t X_s ds\right)\left(X_0^\top - X_{-\delta}^\top + \int_{-\delta}^ 0 X_s^\top ds H\right). 
\end{equation*}
Taking expectations yields
\begin{scriptsize}
\begin{equation*}
r_\delta(t) = 2\gamma(t) - \gamma(t+\delta) - \gamma(t-\delta) + \int_{-\delta}^ 0 \gamma(t-s) - \gamma(t-\delta-s) ds H + H \int_{t-\delta}^t \gamma(s) - \gamma(s+\delta) ds + H \int_{t-\delta}^t \int_{-\delta}^0  \gamma(s-u)  duds H,
\end{equation*}
\end{scriptsize}
where the first order terms can be treated with a simple change of variables. For the second order term we obtain that
\begin{footnotesize}
\begin{equation*}
\begin{split}
\int_{t-\delta}^t \int_{-\delta}^0  \gamma(s-u)  duds &= \int_{t-\delta}^t \int_{s}^ {s+\delta} \gamma(x) dx ds = \int_{t-\delta}^t \int_{t-\delta}^ x \gamma(x) ds dx + \int_{t}^ {t+\delta} \int_{x-\delta}^ t \gamma(x) ds dx \\
&= \int_{t-\delta}^t (x-t+\delta) \gamma(x) dx + \int_t^ {t+\delta} (t-x+\delta)\gamma(x) dx.
\end{split}
\end{equation*}
\end{footnotesize}
\end{proof}

\section*{Acknowledgments}
I would like to thank Lauri Viitasaari for his comments and suggestions.

\bibliographystyle{plain}
\bibliography{pipliateekki}
\end{document}